\documentclass[a4paper]{article}
\usepackage{amsmath,amsthm,amssymb,mathtools}
\usepackage[margin=40mm]{geometry}
\title{Second-order cone representable slices of the positive semidefinite cone of size three}
\author{Gennadiy~Averkov\footnote{BTU Cottbus-Senftenberg, Platz der Deutschen Einheit 1, 03046 Cottbus, Germany}}

\newcommand{\cS}{\mathcal{S}}
\newcommand{\R}{\mathbb{R}}

\newcommand{\setcond}[2]{\left\{#1\,:\,#2\right\}}
\newcommand{\sprod}[2]{\left<#1,#2\right>}
\newcommand{\nsizesprod}[2]{\langle#1,#2\rangle}
\newcommand{\tr}{\operatorname{tr}}
\newcommand{\cQ}{\mathcal{Q}}
\newcommand{\cE}{\mathcal{E}}
\newcommand{\Diag}{\operatorname{Diag}}
\newcommand{\SP}{\operatorname{SP}}
\newcommand{\rank}{\operatorname{rank}}

\newtheorem{thm}{Theorem}

\newtheorem{cor}{Corollary}
\newtheorem{prop}{Proposition}
\newtheorem{prob}{Problem}
\begin{document}
\maketitle
\begin{abstract}
	To demonstrate the discrepancy between second-order cone and semidefinite programming, Hamza Fawzi showed that the cone $\cS_+^3$ of symmetric positive semidefinite matrices of size $3$ is not second-order cone representable (socr). A slice of $\cS_+^3$ is intersection of $\cS_+^3$ and a linear sub-space of the space $\cS^3$ of $3 \times 3$ symmetric matrices. It is known that some  slices of $\cS_+^3$ are socr, while some others are not. We classify socr slices of $\cS_+^3$ by showing that a slice of $\cS_+^3$ is socr if and if it has dimension at most $4$ or is orthogonal to a non-zero singular matrix (where the orthogonality is considered with respect to the standard trace scalar product).  
\end{abstract}
\section{Introduction}

 Second-order cone and semidefinite programming are two prominent examples of conic optimization paradigms beyond linear programming \cite{MR1857264}. The former is a special case of the latter. Glineur, Sounderson and Parrilo \cite{GSP:talk:2013} provided examples of semidefinite constraints of size $3$ that can be lifted to the second-order constraints. On the other hand, using the slack-matrix criterion from \cite{GPT:2013} and combinatorial arguments, Fawzi \cite{Fawzi:2018} showed that it is not possible to reduce semidefinite programming to second-order cone programming by means of lifting, since already the cone of symmetric positive semidefinite matrices of size $3$ does not admit a second-order cone lifting (see also \cite{averkov2019optimal} and \cite{saunderson2019limitations} for ramifications and generalizations of Fawzi's result in the context of polynomial optimization). The overall picture however is far from being complete, as one does not know in general what kind of semidefinite constraints are reducible to second-order cone constraints. In this note we characterize semidefinite constraints of size $3$ that can be lifted to second-order cone constraints. 
  
  For a vector space  $V$ over $\R$ and a convex cone $C \subseteq V$, we say that a set $S \subseteq C$ is a \emph{slice} of $C$ if $S$ is the intersection of $C$ and a linear subspace of $V$. We say that a set $S$ has a $C$-\emph{lift}  if $S$ is the image of a  slice of $C$ under a linear map. 
 
 Semidefinite optimization is conic optimization with respect to the cone $\cS_+^k$ of symmetric positive semidefinite ($=$ psd) matrices of size $k$ in the vector space $\cS^k$ of $k \times k$ symmetric matrices. We introduce the standard Euclidean structure in $\cS^k$ through the \emph{trace scalar product} $\sprod{A}{B} := \tr(AB)$. Each $B \in \cS^k$ determines the slice 
 \[
	 S_B := \setcond{A \in \cS_+^k}{\sprod{A}{B}=0}.
 \]
of $\cS_+^k$ orthogonal to $B$. 
 
Second-order cone optimization is conic optimization with respect to Cartesian powers
 \[
 	\cQ^m := \underbrace{Q \times \cdots \times \cQ}_{m}.
 \] 
 of the \emph{three-dimensional Lorentz cone}
 \[
 \cQ := \setcond{x \in \R^3}{x_3 \ge \sqrt{x_1^2 + x_2^2}}.
 \]
 
 $\cQ$ is linearly isomorphic to $\cS_+^2$ due to the equality
 \begin{equation} \label{Q:as:psd:2}
 	\cQ = \setcond{ x \in \R^3}{\begin{pmatrix*}[c] x_3 -x_1 & x_2 \\ x_2 & x_3 + x_1 \end{pmatrix*} \in \cS_+^3}.
 \end{equation}
Thus, $\cQ^m$ is linearly isomorphic to $(\cS_+^2)^m$ and by this also to the slice of $\cS_+^{2m}$ consisting of block-diagonal psd matrices with $m$ blocks of size $2$.
This shows that second-order cone optimization is a special case of semidefinite optimization.  

A closed convex cone $C$ is said to be \emph{second-order cone representable ($=$ socr)} if $C$ has a  $\cQ^m$-lift for some $m$.  Given a positive integer $k$, we are interested in the problem of characterization of all socr slices of $\cS_+^k$. Our main result is the complete solution in the smallest non-trivial case $k=3$:

 \begin{thm} \label{main} A slice $S$ of $\cS_+^3$ is socr if and only if $\dim(S) \le 4$ or $S=S_B$ for some $B \in \cS^3 \setminus \{0\}$ with $\det(B) =0$.
 \end{thm}

Slices of dimension at most $2$ are not very interesting. They are just polyhedral cones positively spanned by at most two vectors. Rather surprisingly, Theorem~\ref{main} says that \emph{all} slices of $\cS_+^3$ of dimensions $3$ and $4$ are socr. The theorem also allows to give an explicit description of $5$-dimensional socr slices: 

\begin{cor} \label{dim5}
	$\setcond{S_B}{B \in \cS^3 \ \text{singular and indefinite}}$ is the set of all socr slices of $\cS_+^3$ of dimension $5$. 
\end{cor}

Corollary~\ref{dim5} yields a semi-algebraic description of $5$-dimensional socr slices of $\cS_+^3$ via the equivalence:
\begin{align}
	& B  \in \cS^3 \ \text{singular and indefinite} \ \Leftrightarrow \nonumber 
	\\   & \det(B)  = 0 \ \text{and} \ \det(B_{\{1,2\}})+ \det(B_{\{1,3\}}) + \det(B_{\{2,3\}}) > 0. \label{det:descr}
\end{align}
Here, $\det(B_I)$ denotes the principal minor indexed by $I$. 

Examples of  $5$-dimensional slices $S_B$ that have previously been considered are listed in Table~\ref{ex}. 

\begin{table}[htb]
\begin{center}
\begin{tabular}{lll}
	\hline 
	\textbf{$5$-dimensional slices of $\cS_+^3$}  & \textbf{socr} & \textbf{source}
	\\ 
	\hline 
	$S_1:=\setcond{(a_{ij}) \in \cS_+^3}{a_{11}=a_{22}}$  
	& yes & \cite{GSP:talk:2013}
	\\
	$S_2:=\setcond{(a_{ij}) \in \cS_+^3}{a_{11}=a_{22}+a_{33}}$ 
	& no & \cite{Fawzi:arxiv:2016}
	\\ 
	$S_3:=\setcond{(a_{ij}) \in \cS_+^3}{a_{22}=a_{13}}$
	& no & \cite{Fawzi:2018}
	\\
	\hline
\end{tabular}
\end{center}
\caption{\label{ex} Examples of $5$-dimensional slices of $\cS_+^3$. Note that, while \cite{Fawzi:arxiv:2016} is the Arxiv version of \cite{Fawzi:2018}, $S_2$ is considered only in \cite{Fawzi:arxiv:2016} and $S_3$ is considered only in \cite{Fawzi:2018}. Example $S_1$ from the unpublished source \cite{GSP:talk:2013} is also explained in \cite{Fawzi:arxiv:2016,Fawzi:2018}.}
\end{table}  

Regarding $S_1$ from Table~\ref{ex}, we mention that the argument of Glineur, Sounderson and Parrilo \cite{GSP:talk:2013} shows that $S_1$ is the linear image of $\cQ^2$. \cite{GSP:talk:2013} contains a verification of the equivalence: 
\begin{align}
	\begin{pmatrix}
		t & a & b 
		\\ a & t & c 
		\\ b & c & s
	\end{pmatrix} \in \cS_+^3 & &\Leftrightarrow & & \exists u \in \R : \ (b+c,t+a,u, b-c,t-a, 2s -u) \in \cQ^2, \label{GSP:eq}
\end{align}
This can also be described as the equality
\begin{align}
	\setcond{(a_{ij}) \in \cS_+^3}{a_{11}=a_{22}} = \setcond{ \begin{pmatrix} x_2 + y_2 & x_2 - y_2 & x_1 + y_1 
		\\ x_2 -y_2 & x_2 + y_2 & x_1 -y_1 
		\\ x_1 + y_1 & x_1-y_1 & x_3 + y_3 \end{pmatrix}  }{(x,y) \in \cQ^2}. \label{eq:spec:slice} 
\end{align}

As a byproduct of the proof of Theorem~\ref{main} and  \eqref{eq:spec:slice}, we are able to determine an $m$, for which all socr slices of $\cS_+^3$ have a $\cQ^m$-lift.
\begin{cor} \label{Q2} {\ }
	\begin{itemize}
		\item[(a)] Every slice of $\cS_+^3$ of dimension at most $4$ admits a $\cQ^2$-lift.
		\item[(b)] If a $5$-dimensional slice $S$ of $\cS_+^3$ is socr, then $S$ is the image of $\cQ^2$ under a linear map.
	\end{itemize}
\end{cor}

We also shortly discuss affine slices of $\cS_+^k$. We define an \emph{affine slice} of a convex cone $C \subseteq V$ in a vector space $V$ as the intersection of $C$ with an affine subspace of $V$. When dealing with affine slices of $\cS_+^k$, one often uses the notions of linear matrix inequality ($=$ LMI) and spectrahedron. For a linear map $A : \R^n \to \cS^k$ and $B \in \cS^k$, the constraint $A(x) + B \in \cS_+^k$ is called an LMI of size $k$. The respective set 
\[
	\SP_{A,B}:=\setcond{x \in \R^n}{A(x) +B \in \cS_+^k}
\] 
given by this LMI is called a \emph{spectrahedron} \cite{blekherman2012semidefinite}. Modulo the lineality space (the vector space consisting of all vectors parallel $u$ that are parallel to a line contained in $\SP_{A,B}$) every spectrahedron $\SP_{A,B}$ is isomorphic to an affine slice of $\cS_+^k$. As a consequence of Theorem~\ref{main}, we obtain: 

\begin{cor} \label{spec3}
	Every spectrahedron defined by an LMI of size $3$ and of dimension at most $3$ is affinely socr. 
\end{cor}

Arguably, the most well known spectrahedron is the three-dimensional \emph{eliptope} \cite{laurent1995positive}
\begin{align*}
\cE_3:=
 \setcond{x \in \R^3}{
\begin{pmatrix}
1 & x_1 & x_2
\\ x_1 & 1 & x_3
\\ x_2 & x_3 & 1
\end{pmatrix} \in \cS_+^3}.
\end{align*}
$\cE_ 3$ is affinely socr. In view of Corollary~\ref{spec3}, $\cE_3$ owes this property solely to its dimension and the size of the respective LMI. 

In \cite{averkov2019optimal} and \cite{saunderson2019limitations}, the authors studied if, for a given $m$, certain cones arising in polynomial optimization have a $(\cS_+^m)^n$-lift for some integer $n$. One can formulate the same problem for slices of $\cS_+^k$:
\begin{prob}
For given integers $0 < m < k$, characterize slices of $\cS_+^k$ that have a $(\cS_+^m)^n$-lift for some integer $n$. 
\end{prob}

In view of \eqref{Q:as:psd:2}, Theorem~\ref{main} settles the case $m=1,k=2$ of this problem. All other cases seem to be open. 

\section{Proofs}

Throughout, $k,m$ and $n$ denote positive integers. $\Diag(a_1,\ldots,a_k)$ stands for the $k \times k$ diagonal matrix with the diagonal entries $a_1,\ldots,a_k$ in this order. The transpose of a matrix $A$ is denoted by $A^\top$. In matrix expressions, elements of $\R^n$ are interpreted as columns.

We first formulate basic propositions on the geometry of the semidefinite cone $\cS_+^k$.  

\begin{prop}[see \cite{barvinok2002course}] \label{faces}
	 $\setcond{S_B}{B \in \cS_+^k}$ is the set of all faces of $\cS_+ ^k$. If $B \in \cS_+^k$ has rank $r$, the face $S_B$ is linearly isomorphic to $\cS_+^{k-r}$ and has dimension $(k-r)(k-r+1)/2$.  
\end{prop}

\begin{prop} \label{obs}
	For a slice $S$ of $\cS_+^k$ with $k \ge 2$, the following conditions are equivalent:
	\begin{itemize}
		\item[(i)] $S$ has co-dimension $1$.
		\item[(ii)] $S=S_B$ for some indefinite matrix $B \in \cS^k$. 
	\end{itemize} 
\end{prop}
\begin{proof}
	\emph{(i) $\Rightarrow$ (ii):} If $S$ is a slice of co-dimension one, we can represent it as  $S=S_B$ with $B \in \cS^k \setminus \{0\}$. If $B$ were (positive or negative) semidefinite, then, possibly interchanging the roles of $B$ and $-B$, we could assume that $B$ is positive semidefinite. Proposition~\ref{faces} would yield $\dim(S_B) < \dim(\cS_+^k) -1$, which is a contradiction. Thus, $B$ is indefinite.
	
	\emph{(ii) $\Rightarrow$ (ii):} If $B$ is indefinite, then there exist $x, y \in \R^n$ with $x^\top B x >0$ and $y^\top B y < 0$. This yields $x^\top B x = \sprod{A}{B} > 0$ for $A = x x^\top \in \cS_+^k$ and $y^\top B y = \sprod{A}{B} < 0$ for $A = y y^\top \in \cS_+^k$. Thus, the hyperplane $L=\setcond{A \in \cS^k}{\sprod{A}{B}=0}$ splits $\cS^k$ into two open half-spaces, both containing points of the full-dimensional cone $\cS_+^k$. This yields (ii).
\end{proof}

\begin{proof}[Proof of Theorem~\ref{main}]
	Let $S$ be a slice of $\cS_+^3$. One has $\dim(S) \le \dim(\cS_+^3) = 6$. We first derive a number assertions in cases that depend on $\dim(S)$. 
	
	\emph{Case~1: $\dim(S)=6$}. In this case $S= \cS_+^3$ is not socr by Fawzi's result \cite{Fawzi:2018}. 
	
	\emph{Case~2: $\dim(S)=5$.} By Proposition~\ref{obs}, $S=S_B$ holds for some indefinite $B \in \cS^3$. We can factorize $B$  as $B = M^\top D M$, where $M$ is a regular matrix and $D$ is a non-zero diagonal matrix with the diagonal entries belonging to $\{-1,0,1\}$. 
	
	Recall  that $\sprod{A}{B} =\tr(AB) = 0$ and that $\tr(XY)=\tr(YX)$ holds for all $X, Y \in \R^{n \times n}$. We thus have 
	\begin{align}
		\sprod{A}{B} & = \tr(AB) & & \text{(use $B=M^\top D M$)} \nonumber
		\\ & = \tr(A M^\top D M) & & \text{(move factor $M$ from right to left)} \nonumber
		\\ & = \tr(M A M^\top D) & & \text{(write as a scalar product with $D$)} \nonumber 
		\\ & = \nsizesprod{\underbrace{M A M^\top}_{=:F(A)}}{D}. \label{F(A)D}
	\end{align}
	The map $F : \cS^3 \to \cS^3$ is a linear bijection satisfying $F(\cS_+^3) = \cS_+^3$. Thus,
	\begin{align*}
		S  = S_B
		  & = \setcond{A \in \cS_+^3}{\sprod{A}{B}=0} & & \text{(use \eqref{F(A)D})}
		\\ & = \setcond{A \in \cS_+^3}{\sprod{F(A)}{D}=0} & & \text{(exchange $A$ with $F^{-1}(A)$)}
		\\ & = F^{-1} (\setcond{A \in \cS_+^3}{\sprod{A}{D} =0})
		\\ & = F^{-1} (S_D).
	\end{align*} 
 This implies that the slice of $S_B$ is linearly isomorphic to $S_D$ so that we can deal with $S_D$ rather than $S_B$. Since $B$ is indefinite, the matrix $D$ is indefinite, too. Hence, $D$ contains both $1$ and $-1$ on its diagonal. Thus, after possibly exchanging the roles of $D$ with $ - D$ (which does not affect $S_D$) and ordering diagonal entries, we arrive at the cases $D=\Diag(1,-1,0)$ and $D=\Diag(1,-1,-1)$.
	In the former case, $D$ is singular and the respective slice $S_D$ is socr, because $S_D=S_1$ with $S_1$ from Table~\ref{ex}, while in the latter case $D$ is non-singular and the respective slice $S_D$ is not socr, because $S_D=S_2$ with $S_2$ from Table~\ref{ex}. Since the map $F$ preserves singularity of matrices, we conclude that $S=S_B$ is socr if and only if the matrix $B$ is singular. 
	
	\emph{Case~3: $\dim(S) \le 4$.} 
	
	\emph{Subcase~3a: $S$ does not contain interior points of $\cS_+^3$}. Using the fact that $S$ and $\cS_+^k$ can be separated by a hyperplane, we conclude that $S$ is a slice of a face $F$ of $\cS_+^3$ with $F \varsubsetneq \cS_+^3$. Then, in view of Proposition~\ref{faces}, $S=S_B$ for some non-zero semidefinite matrix $B$. Since $S_B$ is linearly isomorphic to $\cS_+^{3-r}$ for $r:=\rank(B)>0$, we conclude that $S=S_B$ is socr.
	
	\emph{Subcase~3b: $S$ contains interior points of $\cS_+^3$.}  The linear hull of $S$, which we denote by $L$, has the same dimension as $S$ and we can represent $S$ as $S = \cS_+^3 \cap L$. 
	
	Consider the orthogonal complement 
	\[
		L^\perp := \setcond{B \in \cS^3}{\sprod{B}{A}=0 \ \text{for all} \ A \in L}
	\] of $L$, which is a space of dimension
	\begin{align*}
		\dim(L^\perp) = \dim(\cS^3) - \dim(L) = 6 - \dim(L) \ge 2. 
	\end{align*}	
	 For every $B \in L^\perp$, we have $S \subseteq S_B$.  Observe that every $B \in L^\perp \setminus \{0\}$ is indefinite, as otherwise $S \subseteq S_B$ with a non-zero semidefinite $B$ and Proposition~\ref{faces} would yield that $S$ does not contain interior points of $\cS_+^3$. Let us fix an arbitrary $C \in L^\perp \setminus \{0\}$. If $C$ is singular, then the assertion of Case~2 yields that $S_C$ is socr. But then, since $S$ is a linear slice of $S_C$, we conclude that $S$ is socr, too. 
	
	If $C$ is non-singular, we  consider a path $\Gamma$ in $L^\perp \setminus \{0\}$ that connects $C$ with $-C$. Such a path exists because $\dim(L^\perp) \ge 2$, which implies that the set $L^\perp \setminus \{0\}$ is connected. After possibly exchanging the roles of $C$ and $-C$, we assume that $C$ has two positive and one negative eigenvalue (counting multiplicities). Then $-C$ has two negative and one positive eigenvalue. 
	 
	 We claim that by letting a matrix $B$ move along the path $\Gamma$ from $C$ to $-C$, we will encounter a singular matrix $B$. We use the notation $\lambda_1(B) \le \lambda_2(B) \le \lambda_3(B)$ to denote the three eigenvalues of $B \in \cS^3$ listed in the increasing order and counting multiplicities. It is well known that the spectrum of a matrix $B$ of a given size is continuous in $B$. Within the space $\cS^3$ of $3 \times 3$ symmetric matrices,  this means that $\lambda_1,\lambda_2,\lambda_3 : \cS^3 \to \R$ are continuous functions. By the choice of $C$, we have $\lambda_2(C) > 0$ and $\lambda_2(-C) <0$. So, by the intermediate value theorem, $\lambda_2(B)$ attains the value zero at some point $B$ of the path $\Gamma$, which joins $C$ and $-C$.  For such a $B$,  the assertion of of Case~2 yields that $S_B$ is socr. Since $S$ is a slice of $S_B$, we conclude that $S$ is socr, too. 
	 
	 The assertion of the theorem can now be derived from  the assertions of the above cases. 
	 
	 To prove the necessity, we need to show that if a slice $S$ of $\cS_+^3$ is socr and $\dim(S)>4$, then $S =S_B$ for some $B \in \cS_+^3 \setminus \{0\}$ satisfying $\det(B)=0$. The assertion of Case~1 excludes  $\dim(S)=6$. So, $\dim(S)=5$ and the desired $B$ exists by the assertion of Case~2. 
	 
	 As for the sufficiency, if a slice $S$ has dimension at most $4$, then $S$ is socr by the assertion of Case~3. Assume now that $S=S_B$ for some $B \in \cS_+^3 \setminus \{0\}$ with $\det(B)=0$. The matrix $B$ is either indefinite or semidefinite. If $B$ is indefinite, $\dim(S)=5$, by Proposition~\ref{obs}, while the assertion of Case~2 tells us that $S=S_B$ is socr. If $B$ is semidefinite, say positive semidefinite, then in view of $B \ne 0$ and  Proposition~\ref{faces}, $\dim(S)=\dim(S_B) \le 3$ so that the assertion of Case~3 implies that $S$ is socr.
\end{proof}

\begin{proof}[Proof of Corollary~\ref{dim5}]
	This assertion is a direct consequence of Proposition~\ref{obs} and Theorem~\ref{main}. 
\end{proof}

\begin{proof}[Proof of Corollary~\ref{Q2}] We first prove (b) and then (a). 

	\emph{(b):} If $S$ is $5$-dimensional and socr, then by Corollary~\ref{dim5}, $S=S_B$ for some singular and indefinite $B \in \cS^3$. In the proof of Theorem~\ref{main}, we have seen that such $S_B$ is linearly isomorphic to $S_D$ with $D = \Diag(1,-1,0)$. Thus, (b) follows by applying \eqref{eq:spec:slice} to $S_D$. 
	
	 \emph{(a):} In the proof of Theorem~\ref{main} we have shown that every slice $S$ of $\cS_+^3$ of dimension at most $4$ is either a slice of a face $F$ of $\cS_+^3$ with $F \varsubsetneq \cS_+^3$ or a slice of some $5$-dimensional $S_B$ with an indefinite and singular $B \in \cS^3$. In the former case, the assertion follows from Proposition~\ref{faces}, while in the latter case the fact that $S_B$ is socr implies that $S$ is socr, too. 
\end{proof}

\begin{proof}[Proof of Corollary~\ref{spec3}]
	Consider an arbitrary spectrahedron
	\[
		\SP_{A,B} = \setcond{x \in \R^n}{A(x) +B \in \cS_+^3}
	\]
	given by an LMI $A(x) + B \in \cS_+^3$ of size $3$ with $n \le 3$. The spectrahedral cone
	\[
		C = \setcond{ (x,y) \in \R^n \times \R}{ A(x) + y B  \in \cS_+^3}
	\]
	is a "homogeneous version" of $\SP_{A,B}$.
	Modulo the lineality space, $C$ is isomoprhic to a slice of $\cS_+^3$. As $\dim(C) \le \dim (A(\R^n)) +1 \le n+1 \le 4$, the respective slice is of dimension at most $4$. So, by Theorem~\ref{main},  $C$ is socr. Using $\SP_{A,B}  = \setcond{x \in \R^n}{(x,1) \in C}$, we conclude that $\SP_{A,B}$ is affinely socr. 
\end{proof}

\paragraph*{Acknowledgments.} This research was inspired by a visit of the Convexity Day at Max-Planck-Institut f\"ur Mathematik in den Naturwissenschaften in Leipzig, organized by Rainer Sinn, Bernd Sturmfels and Thomas Wannerer on September 9, 2019. The overview talk of Daniel Plaumann in this meeting, referring to genericity results from \cite{ottem2015quartic}, has led me to considering generic lifting properties of $n$-dimensional slices of $\cS_+^k$. 

\bibliographystyle{amsalpha}
\bibliography{literature}

\end{document}